\documentclass{amsart}
\usepackage{graphicx}
\newtheorem{thm}{Theorem}[section]
\newtheorem{cor}[thm]{Corollary}

\newtheorem{prop}[thm]{Proposition}
\theoremstyle{definition}
\newtheorem{defn}[thm]{Definition}
\theoremstyle{remark}
\newtheorem{rem}[thm]{Remark}
\numberwithin{equation}{section}

\begin{document}

\title[Invariant subspaces in the polydisc]
 {Invariant subspaces generated by a single function in the polydisc}

\author[Koca]{Beyaz Ba\c{s}ak Koca}
\address{Istanbul University Faculty of Science Department of Mathematics Istanbul, Turkey}
\email{basakoca@istanbul.edu.tr}

\author[Sadik]{Nazim Sadik}
\address{Istanbul, Turkey} \email{sadnaz@mail.ru}



\begin{abstract}
In this study, we partially answer the question left open in Rudin's
book \cite{rudin} on the structure of invariant subspaces of the Hardy space $H^2(U^n)$ on the polydisc $U^n$. We completely describe all
invariant subspaces generated by a single function in the polydisc. Then, using our
results, we give the unitary equivalence of this type of invariant
subspace and a characterization of outer functions in $H^2(U^n)$.
\end{abstract}

\maketitle
\section{Introduction}
In his book \cite[p.78]{rudin}, Rudin posed the question: ``One may
ask for a classification or an explicit description (in some sense)
of all invariant subspaces of $H^2(U^n)$.'' For $n=1$,  Beurling
showed, in his well-known paper \cite{beurling},  that every
invariant subspace $M$ of the Hardy space $H^2(U)$ on the unit disc $U$
is of the form $M=fH^2(U)$ for some inner function $f$, i.e., is
generated by a single inner function. This type of invariant
subspace is called the Beurling type. For $n\geq 2$, the structure
of invariant subspaces of $H^2(U^n)$ is much more complicated. It is
clear that the subspaces generated by an inner function are also
invariant in $H^2(U^n)$. However, it follows from a result of
Jacewicz \cite{jacewicz} that there is an invariant subspace of
$H^2(U^2)$ that is generated by two functions and cannot be
generated by any single function. In addition, Rudin \cite{rudin}
also showed that there are invariant subspaces of $H^2(U^2)$ that
are not even finitely generated. One may ask what additional
conditions are required for an invariant subspace to be generated by
a single function? Sadikov \cite{sadikinv} described the Beurling
type invariant subspace of $H^2(U^n)$ in terms of the classical
theorem of Beurling-Lax-Halmos (see, \cite[Corollary 3.26,
p.53]{radjavi}) by considering the space $H^2(U^n)$ as the
vector-valued analytic functions on the unit disc. Agrawal, Clark,
and Douglas \cite{agrawal} gave a necessary and sufficient condition
for Beurling type invariant subspaces in terms of the unitary
equivalence of invariant subspaces, while Mandrekar \cite{mandrekar}
gave a necessary and sufficient condition in terms of the double
commutativity of shifts. More generally, Sarkar, Sasane, and Brett
\cite{sarkar} showed that an invariant subspace of vector-valued
Hardy space satisfying the doubly commuting property has a
Beurling-Lax-Halmos representation.

In this study, we consider the following questions:
\begin{enumerate}
\item[1)] What is the complete characterization of the function $f\in H^\infty(U^n)$
for the subspace $fH^2(U^n)$ of $H^2(U^n)$ to be invariant?
\item[2)] Is every singly generated invariant subspaces of $H^2(U^n)$
Beurling type?
\end{enumerate}
These questions are answered by Theorems \eqref{first} and
\eqref{second}, respectively. Then, we study some properties of
singly invariant subspaces (inclusions and unitary
equivalence). Moreover, a characterization of outer functions in
$H^2(U^n)$ is given as in the one-variable case.

Before beginning, let us recall some required facts.

Throughout this paper,  $n$ is a positive integer, and
$\mathbb{C}^n$ is the vector space of all ordered $n$-tuples
$z=(z_1,\ldots,z_n)$ of complex numbers with inner product $\langle
z,w\rangle=\sum z_i\bar w_i$, norm $|z|=\langle z,z\rangle^{1/2}$,
and corresponding polydisc
\[U^n=\{z\in \mathbb{C}^n: |z_i|<1,\;\; i=1,\ldots,n\},\]
whose distinguished boundary is the torus
\[T^n=\{z\in \mathbb{C}^n: |z_i|=1,\;\; i=1,\ldots,n \}.\]
If $f$ is in $L^1(T^n)$, the space of all functions integrable with
respect to normalized Lebesgue measure $m_n$ on $T^n$, then its
Poisson integral $P[f]$ is the function
\[P[f]=\int_{T^n}P(z,w)f(w)dm_n(w),\;\; z\in U^n,\]
where $P(z,w)$ is the Poisson kernel.

If $f$ is any function in $U^n$ and $0\leq r <1$, $f_r$ will denote
the function defined on $T^n$ by $f_r(w)=f(rw)$. Then the function
$f^*$ is defined as
\[f^*(w)=\lim_{r\rightarrow 1}f_r(w)\] at every $w\in T^n$ at which
this radial limit exists.

The Hardy space on the polydisc $H^2(U^n)$ is defined as the class
of all holomorphic functions $f$ in $U^n$  for which
\[||f||_2=\sup_{0\leq r<1}\left\{\int_{T^n}|f_r|^2dm_n\right\}^{1/2}<\infty.\]
It is well known that if $f\in H^2(U^n)$, then the radial limit
\[f^*(w)=\lim_{r\rightarrow 1}f_r(w)\]
exists for  almost all $w\in T^n$, and
\[\lim_{r\rightarrow 1}\int_{T^n}|f_r-f^*|^2dm_n=0.\]
As usual, we also treat $H^2(U^n)$ as a closed subspace of
$L^2(T^n)$, the standard Lebesgue space of square-integrable
functions on $T^n$. $H^\infty(U^n)$ is the
space of all bounded holomorphic functions  $f$ in $U^n$;
$||f||_\infty=\sup_{z\in U^n}|f(z)|<\infty$. An inner function in
$U^n$ is a function $f\in H^\infty(U^n)$ with $|f^*|=1$ a.e. on $T^n$. A function $f\in
H^2(U^n)$ is said to be outer if
\[\log |f(0)|=\int_{T^n}\log |f^*|dm_n.\]
Recall that a subspace $M$ of $H^2(U^n)$ is called ``invariant'' if (a) $M$ is a
closed linear subspace of $H^2(U^n)$ and (b) $f\in M$ implies
$z_if\in M$ for $i=1,\ldots,n$; i.e., multiplication by the
variables $z_1,\ldots,z_n$ maps $M$ into $M$. The smallest invariant subspace of $H^2(U^n)$ which contains a given $f$ is denoted by $M(f)$ and $M(f)$ is called the subspace generated by $f$ if $M(f)=fH^2(U^n)$.

The following theorem is required in this study.
\begin{thm}{\cite[Theorem 3.5.3, p.55]{rudin}}\label{boundary}
If $\psi$ is positive, bounded, and lower semi-continuous on $T^n$,
then $\psi=|f^*|$ a.e., for some $f\in H^\infty(U^n)$.
\end{thm}
For further information on Hardy space on the polydisc, see
\cite{rudin}.

\section{Main Results}
It is clear that the subspace $\varphi H^2(U^n)$ of $H^2(U^n)$ with
some inner function $\varphi$ is invariant. It is natural to ask
whether there exists any function $f$ that is not necessarily inner
such that $fH^2(U^n)$ is invariant. The following theorem describes
the class of all such functions.
Before starting, we give the following definition.
\begin{defn}
  A function $f\in H^\infty(U^n)$ with $f^{-1}\in L^\infty(T^n)$ is called a generalized inner function.
\end{defn}
Here, it is clear that $f^{-1}=1/f^\ast$.
\begin{thm}\label{first}
Let $f\in H^\infty(U^n)$. The subspace $f H^2(U^n)$ of $H^2(U^n)$ is invariant if and only if $f$ is a generalized inner function.
\end{thm}
\begin{proof}
Let $M_f$ denote the bounded linear operator on $H^2(U^n)$ given by
$M_f(g)=fg$ for any $g\in H^2(U^n)$. Suppose that the subspace $f
H^2(U^n)$ is invariant. Since $\mbox{Ker}M_f=\{0\}$ and the image of
$M_f$, $f H^2(U^n)$ is closed, $M_f$ is bounded below, that is,
there exists a number $\delta>0$ such that $||M_fg||_2\geq
\delta||g||_2$ for any $g\in H^2(U^n)$. Then $|f|\geq \delta$ a.e.
on $T^n$. In fact, assume that $m_n\{\xi\in T^n: |f|<\delta\}>0$.
Then $m_n\{\xi\in T^n: |f|<\delta_0\}>0$ for some
$\delta_0\in(0,\delta)$. Let us fix such a $\delta_0$ and put
$E=\{\xi\in T^n:|f(\xi)|<\delta_0\}$. We can construct a sequence of
continuous functions $\{\varphi_j\}_{j\geq1}$ defined on $T^n$ such
that $0<\varphi_j\leq 1$ and
$\lim_{j\rightarrow\infty}\varphi_j=\chi_E$ a.e. on $T^n$, where
$\chi_E$ denotes the characteristic function of $E$. By Theorem
\eqref{boundary} for each $j$ there exists a function $g_j\in
H^\infty(U^n)\subset H^2(U^n)$ such that $|g_j|=\varphi_j$ a.e. on
$T^n$. We get
\[\delta^2\int_{T^n}|g_j|^2dm_n\leq\int_{T^n}|f|^2|g_j|^2dm_n\]
for all $j$. Applying the Lebesgue dominated convergence theorem we
have
\[\delta^2m_n(E)\leq \int_E|f|^2dm_n\leq \delta_0^2m_n(E)<\delta^2m_n(E),\]
and we get a contradiction. Hence $|f|\geq \delta$ a.e. on $T^n$,
i.e., $f^{-1}\in L^\infty(T^n)$. Conversely, suppose that $f^{-1}\in
L^\infty(T^n)$. It is clear that  for the invariance of $fH^2(U^n)$,
it is enough to show the closedness of the subspace $f H^2(U^n)$.
For this, we show that $M_f$ is bounded below, that means there
exists a number $c>0$ such that $||M_fg||_2\geq c||g||_2$ for any
$g\in H^2(U^n)$. In fact, for any $g\in H^2(U^n)$ we obtain
\[||g||_2=||f^{-1}f g||_2\leq ||f^{-1}||_\infty\cdot||fg||_2=||f^{-1}||_\infty\cdot||M_fg||_2,\]
and the proof is complete.
\end{proof}
Remark that this theorem shows that every singly generated invariant subspaces of $H^2(U^n)$ is generated by a generalized inner function.
It is clear that the class of inner functions is properly contained
in the class of generalized inner function. 
It is
natural to ask whether every singly generated invariant subspaces is generated by an inner function, i.e., is Beurling type.
The following construction leads to an affirmative answer:
\begin{thm}\label{second}
There exists a generalized inner function $f$ such that $f H^2(U^n)\neq  IH^2(U^n)$ for any inner
function $I$.
\end{thm}

\begin{proof}
We can take a function $h\in C(T^n)$ such that $h>0$ everywhere and
the Poisson integral of $\log h$ is not pluriharmonic. By Theorem
\eqref{boundary}, there exists a function $f\in H^\infty(U^n)$ such
that $|f^*|=h$ almost everywhere on $T^n$. Suppose that $f
H^2(U^n)=IH^2(U^n)$ for some inner function $I$. Then, $fI^{-1}
H^2(U^n)=H^2(U^n)$, hence $fI^{-1}$ is outer by \cite[p.74, Theorem
4.4.6]{rudin}. It follows that the equality
\[\mbox{Re}\log(fI^{-1})=\log|fI^{-1}|=P[\log|(fI^{-1})^*|]=P[\log|f^*|]=P[\log h],\]
is satisfied, showing that $P[\log h]$ is pluriharmonic, and we get
a contradiction, and the proof is complete.
\end{proof}

\begin{rem}
By the proof of Theorem \eqref{second}, we have the function $f\in
H^\infty (U^n)$, $f\not\equiv 0$ and $|f^*|=h$ a.e. on $T^n$, where $h$ is
continuous on $T^n$. Then, by \cite[Theorem 5.4.5, p. 121]{rudin},
$f$ has the same zeros as some inner function $u$. This means that
$f=f_1u$, $f_1$ is holomorphic in $U^n$, and $f_1$ has no zero in
$U^n$. However, since $P[\log|f^*|]=P[\log h]$ is not the real part
of any holomorphic function in $U^n$, $f$ can not be written as the product of an
inner function and an outer function by \cite[Exercise 5.5.7 (a),
p.130]{rudin}.
\end{rem}

\begin{prop}\label{prop}
Let $f_1$ and $f_2$ be two generalized inner functions. The invariant subspaces $f_1H^2(U^n)$ and $f_2
H^2(U^n)$ satisfy the following conditions:
\begin{enumerate}
\item[(a)] $f_1H^2(U^n)\subset f_2 H^2(U^n)$ if and only if
$f_1f_2^{-1}\in H^\infty(U^n)$ with $f_2f_1^{-1}\in L^\infty(T^n)$.
\item[(b)] $f_1H^2(U^n)= f_2 H^2(U^n)$ if and only if
$f_1f_2^{-1}, f_2f_1^{-1}\in H^\infty(U^n)$.
\end{enumerate}
\end{prop}

\begin{proof}
\begin{enumerate}
\item[(a)] Assume that $f_1H^2(U^n)\subset f_2 H^2(U^n)$. Then,
$f_1\in f_2 H^2(U^n)$ and thus $f_1=f_2g$ for some $g\in H^2(U^n)$.
Moreover, $f_1\in H^\infty(U^n)$ and $f_2^{-1}\in L^\infty(T^n)$
gives $g= f_1f_2^{-1}\in H^\infty(U^n)$. On the other hand, if
$f_1^{-1}\in L^\infty(T^n)$, then $g^{-1}=f_1^{-1}f_2\in
L^\infty(T^n)$. Conversely, if $f_1=f_2h$, where $h\in
H^\infty(U^n)$ with $h^{-1}\in L^\infty(T^n)$, then $f_1f=f_2(hf)$
for any $f\in H^2(U^n)$ and therefore $f_1H^2(U^n)\subset f_2
H^2(U^n)$.
\item[(b)] It is easily seen from (a).
\end{enumerate}
\end{proof}

Beurling's theorem also states (in the one-variable case)
that $f$ is outer if and only if $fH^2(U)=H^2(U)$. Rudin proved that
one part of this theorem holds, if $n>1$ (\cite[Theorem 4.4.6,
p.74]{rudin}: $f$ is outer, if $fH^2(U^n)= H^2(U^n)$), but he also
showed in \cite[Theorem 4.4.8 (b), p.76]{rudin} that there exists an
outer function $f\in H^2(U^2)$ such that $fH^2(U^2)\neq H^2(U^2)$,
which implies that the converse of the theorem fails if $n>1$. The
following result gives a class of functions for which Beurling's
theorem holds for $n>1$.

\begin{cor}\label{outer}
If $f$ is a generalized inner function, then $f$ is
outer if and only if $fH^2(U^n)= H^2(U^n)$.
\end{cor}

\begin{proof}
If $fH^2(U^n)= H^2(U^n)$, then that the function $f$ is outer is
obvious by \cite[Theorem 4.4.6, p.74]{rudin}. Conversely, suppose
that $f$ is outer. Then, by \cite[p.73]{rudin},
\begin{equation}\label{equation}
\log|f|=P[\log|f^*|]
\end{equation}
which implies that $f$ has no zero in $U^n$, hence $f^{-1}$ is
analytic in $U^n$. On the other hand, multiplying equality
\eqref{equation} by $(-1)$, we obtain that the same equality is true
for $f^{-1}$, that is, $\log|f^{-1}|=P[\log|(f^{-1})^*|]$. Since
$(f^*)^{-1}\in L^\infty(T^n)$, $\log|(f^{-1})^*|\in L^\infty(T^n)$,
and then its Poisson integral $P[\log|(f^{-1})^*|]$ is bounded on
$U^n$ by \cite[Theorem 2.1.3, p.18]{rudin}. Hence, $f^{-1}$ is
bounded on $U^n$, and $f^{-1}\in H^\infty(U^n)$. Applying
Proposition \eqref{prop} (b) to $f_1=f$ and $f_2=1$, we have
$fH^2(U^n)= H^2(U^n)$.
\end{proof}

Agrawal, Clark, and Douglas \cite{agrawal} studied the question of
unitary equivalence of invariant subspaces of $H^2(U^n)$. Two
invariant subspaces $M_1$ and $M_2$ are said to be unitarily
equivalent, if there is a unitary operator $U: M_1\rightarrow M_2$
such that $U(\theta f)=\theta(Uf)$, for $\theta\in H^\infty(U^n)$
and $f\in M_1$. Under the assumption of full range of an invariant
subspace $M$ (for definition, see \cite[p.5]{agrawal}), they showed
that all invariant subspaces $N$ of $H^2(U^n)$ unitarily equivalent
to $M$ are of the form $\varphi M$ for some inner function
$\varphi$. However, the Beurling type invariant subspaces do not
have full range \cite[Remark 3, p.6]{agrawal}, and Mandrekar proved
unitary equivalence of this type of invariant subspace in
\cite[Theorem 4]{mandrekar}. In the following theorem, which covers
to Mandrekar's result, we now consider the unitary equivalence of
the singly generated invariant subspaces.

\begin{thm}
\begin{itemize}
\item[(a)] Let $f_1$ and $f_2$ be two generalized inner functions. The invariant subspaces $f_1H^2(U^n)$ and
$f_2H^2(U^n)$ are unitarily equivalent if $|f_1|=|f_2|$ a.e. on
$T^n$.
\item[(b)] Any invariant subspace $M$ of $H^2(U^n)$ unitarily
equivalent to a singly generated invariant subspace is also singly
generated.
\end{itemize}
\end{thm}

\begin{proof}
\begin{itemize}
\item[(a)] Let $|f_1|=|f_2|$ a.e. on $T^n$. Taking the multiplication operator
$M_\psi:f_1H^2(U^n)\rightarrow f_2H^2(U^n)$, where
$\psi:=f_2f_1^{-1}$ -unimodular function, as a desired unitary
operator, we easily obtain unitary equivalence of these subspaces.
\item[(b)] If the invariant subspace $M$ of $H^2(U^n)$ unitarily
equivalent to $fH^2(U^n)$ for some generalized inner function $f$,
then, by Lemma 1 in \cite{agrawal} there exists a unimodular
function $\psi$ such that $M=\psi fH^2(U^n)$ and so, the subspace
$M$ is an invariant subspace generated by $g:=\psi f \in
H^\infty(U^n)$ with $g^{-1}\in L^\infty(T^n)$.
\end{itemize}
\end{proof}

It is well known that every nonzero function $f$ in $H^2(U)$ has a
factorization $f=gh$, where $g$ is inner and $h$ is outer. This is
no longer true for $H^2(U^n)$ for $n>1$ \cite[Theorem 4.2.2.,
p.65]{rudin}. One can naturally asked whether a generalized
statement may hold true, namely, every nonzero function $f$ in $
H^2(U^n)$ can be written as the product $f=gh$, where~$h$ is still
an outer function, but where~$g$ is now a {\em generalized} inner
function. According to \cite[Theorem 4.1.1, p. 60]{rudin}, there
exists an $f$ in $H^2(U^n)$, $f\not\equiv 0$ such that every bounded
holomorphic function vanishing on the zero set of $f$ is equal to 0
identically. Thus, the generalized factorization does not hold for
such $f$.
We do not know whether a generalized factorization holds true for arbitrary bounded~$f\not\equiv 0$ on~$U^n$:\\

\textbf{Problem: } Can every non-zero function $f$ in $
H^\infty(U^n)$ be written as the product $f=gh$, where $g$ is a
generalized inner function, $h$ is an outer function?
\subsection*{Acknowledgements} The authors are grateful to A.B.
Alexandrov for helpful discussions and comments.

\end{document}